\documentclass[hidelinks, 11pt]{article}
\usepackage[a4paper,bindingoffset=0.in,left=2.5cm,right=2.5cm,top=2cm,bottom=3cm]{geometry}
\usepackage{graphicx,subfigure}
\usepackage{url}
\usepackage{enumitem}
\usepackage[utf8]{inputenc}
\usepackage{lmodern}
\usepackage{setspace}
\usepackage[normalem]{ulem}

\usepackage{wrapfig}
\usepackage{mdframed}
\usepackage[a]{esvect}
\usepackage{color}

\usepackage{amsthm}
\newtheorem{theorem}{Theorem}[section]
\newtheorem{proposition}[theorem]{Proposition}

\newtheorem{lemma}[theorem]{Lemma}

\newtheorem{corollary}[theorem]{Corollary}

\theoremstyle{definition} 
\newtheorem{definition}[theorem]{Definition}
\newtheorem{remark}[theorem]{Remark}
\newtheorem{example}[theorem]{Example}
\newtheorem{question}[theorem]{Question}

\usepackage{mathtools, amsmath, amssymb, graphicx, amsfonts}
\usepackage[colorlinks=true,allcolors=blue!75!black,backref=page]{hyperref} 
\usepackage{bbm} 
\usepackage{float} 
\usepackage{amsopn} 
\usepackage{mathrsfs} 
\usepackage{hyperref} 
\usepackage{xcolor} 
\usepackage{comment} 
\renewenvironment{proof}{{\noindent\bfseries Proof.}}{\qed} 

\theoremstyle{theorem}
\newtheorem{innercustomthm}{Theorem}
\newenvironment{customthm}[1]{\renewcommand\theinnercustomthm{#1}\innercustomthm}{\endinnercustomthm}
\newtheorem{innercustomcor}{Corollary}\newenvironment{customcor}[1]  {\renewcommand\theinnercustomcor{#1}\innercustomcor}{\endinnercustomcor}

\newcommand{\tr}{\operatorname{tr}}

\newcommand{\diag}{\operatorname{diag}}

\newcommand{\spn}{\operatorname{span}}
\newcommand{\rank}{\operatorname{rank}}

\newcommand{\R}{\mathbb{R}}

\newcommand{\bgamma}{\boldsymbol{\gamma}}
\newcommand{\Gl}{\mathrm{GL}}

\begin{document}

\title{The geometry of magnitude for finite metric spaces}

\author{Karel Devriendt} 

\date{}

\maketitle
\begin{abstract}
The main result of this article is a geometric interpretation of magnitude, a real-valued invariant of metric spaces. We introduce a Euclidean embedding of a (suitable) finite metric space $X$ such that the magnitude of $X$ can be expressed in terms of the `circumradius' of its embedding $S$. The circumradius is the radius of the unique sphere that goes through $S$. We give three applications: First, we describe the asymptotic behaviour of the magnitude of $tX$ as $t\rightarrow \infty$, in terms of the circumradius. Second, we develop a matrix theory for magnitude that leads to explicit relations between the magnitude of $X$ and the magnitude of its subspaces. Third, we identify a new regime in the limiting behaviour of $tX$, and use this to show submodularity-type results for magnitude as a function on subspaces. 
\\~\\
\textbf{Keywords:} \textit{Magnitude, finite metric spaces, simplex geometry, submodularity}\\
\textbf{AMS subject classification}: \emph{51K05, 15B48, 51F99, 52A38 }
\end{abstract}

\section{Introduction}\label{sec: introduction}

\subsection{Background}
Magnitude is a real-valued invariant for enriched categories (in general) and metric spaces (in particular) that was introduced by Leinster in 2006, in \cite{leinster_2008_euler}. Before giving a definition, we mention two slogans that are helpful in guiding one's intuition on magnitude. The first slogan says that 
$$\text{\emph{``Magnitude is like an Euler characteristic."}}$$
This statement reflects the history of magnitude \cite{leinster_2008_euler}, motivates its definition, and guides the development of its theory, for instance in the generalization to magnitude homology \cite{hepworth_2017_categorifying, leinster_2021_homology}. In this paper, the connection to Euler characteristics shows itself in our study of valuative properties, i.e., how magnitude behaves as a function on subspaces $Y\subseteq X$ and what kind of inclusion--exclusion-type relations it satisfies. A second slogan says that
$$\text{\emph{``Magnitude counts the effective number of points."}}$$
Perhaps more appropriate for finite spaces, this intuition explains the success of magnitude in applications such as quantifying (bio-)diversity \cite{leinster_2016_diversity} and in data analysis \cite{adamer_2024_vector, hunstman_2023_diversity}. The second slogan comes forward in this paper in our study of the asymptotic convergence of magnitude to the point count $\# X$; see also Examples \ref{ex: 2-point} and \ref{ex: 3-point}. At the time of writing, an online bibliography \cite{leinster_2025_bibliography} maintained by Leinster \& Meckes counts more than 120 papers on magnitude. 

\subsection{Summary of results}
In this article, we focus on magnitude for \emph{finite} metric spaces. This setting was treated early on in the development of magnitude, for instance by Leinster in \cite[\S2] {leinster_magnitude_2013} and by Meckes in \cite{meckes_2012_positive}. In particular, both works found that so-called positive definite metric spaces are particularly well-behaved (see Definition \ref{def: positive definite}). Our work builds forth on this analysis by drawing a strong connection between magnitude of finite positive definite metric spaces and Euclidean geometry, via matrix theory. We briefly describe our main result and its applications below.
\\~\\
The magnitude of a finite metric space $(X,d)$ is defined based on the  \emph{similarity matrix} $Z$ with entries $z_{ij}=e^{-d(i,j)}$. When $Z$ is positive definite, we call $X$ \emph{positive definite} and define:
$$
\text{The magnitude of $X=(X,d)$ is~}\vert X\vert := \sum_{i,j\in X} \big(Z^{-1}\big)_{ij}.
$$
A more general definition will
be given later in the paper. The main new idea in this article is that there is a Euclidean embedding $\varphi: X\rightarrow\R^{X-1}$, which we call the \emph{similarity embedding}, such that the magnitude of $X$ can be expressed in terms of the circumradius $R(S)$ of its embedding $S=\varphi(X)$; the circumradius is the radius of the unique sphere in $\R^{X-1}$ that goes through the points $S$. More precisely, we find:
\begin{customthm}{\ref{thm: magnitude circumradius}}
Let $X$ be a positive definite metric space with similarity embedding $S$. Then
\begin{equation}\tag{\ref{eq: main result}}
\vert X\vert \,=\, \frac{1}{1-2R(S)^2}.
\end{equation}
\end{customthm}
In fact, requiring expression \eqref{eq: main result} for all subspaces $Y\subseteq X$ characterizes the embedding $S$ (see Theorem \ref{thm: magnitude circumradius Y}). So, in a sense, the embedding $\varphi$ naturally appears in the context of magnitude.
\begin{example}\label{ex: 2-point}
Let us illustrate Theorem \ref{thm: magnitude circumradius} by a classical example in magnitude: the metric space $X^{(2)}$ consisting of two points at distance $d>0$. By inverting the positive definite similarity matrix $Z=\Big(\begin{smallmatrix}1&e^{-d}\\ e^{-d}&1\end{smallmatrix}\Big)$ and summing the entries of $Z^{-1}$, one directly computes the magnitude
$$
\vert X^{(2)}\vert = 1 + \tanh(d/2).
$$
The embedding of $X^{(2)}$ that we introduce gives two points $S\subseteq \R$ at distance $\sqrt{1-e^{-d}}$. Since the circumradius of $S$ is half this distance, Theorem \ref{thm: magnitude circumradius} gives the relation
$$
\frac{1}{1-2R(S)^2} \,=\, \frac{1}{1-\frac{1-\exp(-d)}{2}} \,=\, \frac{2}{1+\exp(-d)}\,=\, 1+\tanh(d/2) \,=\, \vert X^{(2)}\vert.
$$
\end{example}

What lies behind Theorem \ref{thm: magnitude circumradius}? To lift some of the mystery, let us reveal the underlying linear algebra and construct the embedding $\varphi$ of $X$ explicitly, starting from its matrix $Z$:
\begin{enumerate}[align=left,labelwidth=*,labelindent=2em,leftmargin=!]
\item[Step 1:] Construct the centered matrix $K= \frac{1}{2}(I-\tfrac{\mathbf{11}^T}{n})Z(I-\tfrac{\mathbf{11}^T}{n})$, where $n=\# X$.
\item[Step 2:] Because $Z$ is positive definite, $K$ admits a square root $\sqrt{K}$.
\item[Step 3:] Define the embedding $\varphi(i)$ of $i\in X$ as the $i$th column of $\sqrt{K}$.
\end{enumerate}
In other words, $K$ is the Gram matrix of $S$ when the origin is placed at the centroid of $S$. Based on this embedding, Theorem \ref{thm: magnitude circumradius} then follows from the observation that the circumradius of $S$ is governed by very similar ``equilibrium" equations to those of magnitude; see Remark \ref{rem: equilibrium equations}. 

We continue with a brief description of the three main applications of Theorem \ref{thm: magnitude circumradius}.
\\~\\
\textbf{As a first application}, we consider the asymptotic behaviour of magnitude. Much of the theory of magnitude is concerned with studying the magnitude $\vert t X\vert$ across different scales $t>0$, based on the metric transformation $d\mapsto t\cdot d$. Leinster \& Willerton showed in \cite[Thm. 3]{leinster_2013_asymptotic} that magnitude asymptotically counts the number of points $n$ in the metric space $X$, as follows:
$$
\vert tX\vert = n - q(tX) \text{~with~}q(tX)\rightarrow 0\text{~as~}t\rightarrow \infty.
$$
We give a description of the asymptotics of the error term $q$ in this formula; let $S_t:=\varphi(tX)$.
\begin{customthm}{\ref{thm: asymptotics of q}}
Let $X$ be a metric space on $n$ points. Then we have the asymptotic equivalence
$$n-\vert tX\vert=q(tX) \,\sim\, n^2\left( \frac{n-1}{n} - 2R(S_t)^2\right).$$
\end{customthm}

Geometrically, Theorem \ref{thm: asymptotics of q} states that the magnitude of $tX$ tends to the magnitude of the discrete metric space in the same way as the circumradius of $S_t$ tends to the circumradius of the regular simplex with unit-length edges.  Let us return to the example of the two-point metric space to illustrate the theorem.
\begin{example}
For the two-point metric space $X^{(2)}$, Theorem \ref{thm: asymptotics of q} correctly gives
$$
q(tX^{(2)}) = 2 - \vert t X^{(2)}\vert = \frac{2}{1+\exp(t\cdot d)} \,\,\sim\,\, 4(\tfrac{1}{2}-\tfrac{1}{2}\exp(-t\cdot d)). 
$$
In other words, we find that $\vert tX^{(2)}\vert\approx 2 e^{-t d}$ asymptotically. This is illustrated in Figure \ref{fig: 2point_metric_space}.
\begin{figure}[h!]
\centering
\includegraphics[width=0.42\textwidth]{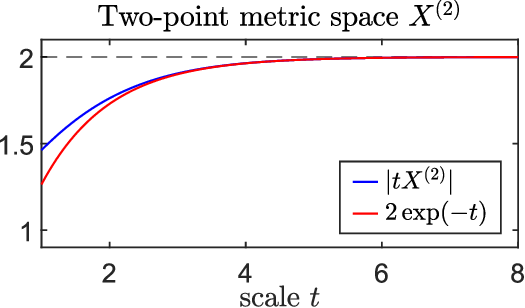}
\caption{The magnitude $\vert tX^{(2)}\vert$ of the metric space with two points at distance $d=1$ is well-approximated by $2e^{-t}$ at large scales $t\gg 0$. This is explained by Theorem \ref{thm: asymptotics of q}.}
\label{fig: 2point_metric_space}
\end{figure}
\end{example}

\textbf{As a second application}, we take inspiration from Fiedler's matrix theory for Euclidean simplices in \cite{fiedler_matrices_2011}, and develop a matrix theory for magnitude. Generalizing the relation between the pseudoinverse Gram matrix and distance matrix of a simplex (see, e.g., \cite[\S3.2]{devriendt_thesis_2022}), we find:
\begin{customthm}{\ref{thm: fiedler-bapat}}
Let $X$ be a metric space with invertible $Z$ and nonzero magnitude. Then
\begin{equation}\tag{\ref{eq: fiedler-bapat}}
\begin{pmatrix}0 & \mathbf{1}^T\\\mathbf{1}& Z\end{pmatrix}^{-1}
\,=\,
\begin{pmatrix}-\vert X\vert^{-1} &\mathbf{w}^T/\vert X\vert \\\mathbf{w}/\vert X\vert& \frac{1}{2}K^\dagger\end{pmatrix}.
\end{equation}
\end{customthm}
In equation \eqref{eq: fiedler-bapat}, the vector $\mathbf{1}=(1,\dots,1)^T$ is the all-ones vector, $\mathbf{w}:=(Z^{-1})\mathbf{1}$ is the \emph{weighting} on the metric space $X$ and $K^\dagger$ is the Moore--Penrose pseudoinverse of the Gram matrix $K$ of the centered point set $S$; more results on the matrices $Z,K$ and $K^\dagger$ are given in Section \ref{sec: matrix theory}. 
\begin{example} For the two-point metric space with distance $d=-\ln(\delta)$, Theorem \ref{thm: fiedler-bapat} reads
$$
\begin{pmatrix}
0&1&1\\
1&1&\delta\\
1&\delta&1
\end{pmatrix}^{-1}\,=\,\,\,
\frac{1}{2}\begin{pmatrix}
-(1+\delta) & 1 & 1\\
1 & (1-\delta)^{-1} & -(1-\delta)^{-1}\\
1 & -(1-\delta)^{-1} & (1-\delta)^{-1}\\
\end{pmatrix}.
$$
\end{example}

What is the use of this identity? In this article, it mainly serves as a tool to study the magnitude of subspaces of $X$. Since any subspace $Y\subseteq X$ of a positive definite metric space is also positive definite, Theorem \ref{thm: fiedler-bapat} applies to $Y$ as well. The subspace relation between $Y$ and $X$ translates nicely to an algebraic relation between their respective matrices, and we arrive at:
\begin{customcor}{\ref{cor: magnitude Y}}
Let $X$ be a positive definite metric space. Then for any $x\in X$ we have
$$
\vert X \backslash\{x\} \vert \,=\, \vert X\vert\left(1 + \frac{2w_x^2}{\bar{c}_x \vert X\vert}\right)^{-1} \quad \textup{~and~}\quad \frac{w'_i}{\vert X\backslash\{x\}\vert} \,=\, \frac{w_i}{\vert X\vert} + \frac{c_{ix}}{\bar{c}_{x}} \frac{w_x}{\vert X\vert},
$$
for any $i\neq x$, and where $\mathbf{w}'$ is the weighting on $X\backslash\{x\}$.
\end{customcor}
Here, we use the notation $c_{ij}=-(K^\dagger)_{ij}$ and $\bar{c}_i=\sum_j c_{ij}$. This is a corollary of Theorem \ref{thm: magnitude Y general}, which is an analogous statement for general subspaces. We illustrate the relation between $\vert X\vert$ and $\vert X\backslash\{x\}\vert$ in Corollary \ref{cor: magnitude Y} by looking at the quantity $\frac{2w_x^2}{\bar{c}_x\vert X\vert}$ that controls this relation. 
\begin{example}
\label{ex: 3-point}
We consider the metric space on $X=\{1,2,3\}$ with distances
$$
d(1,2)=2\quad\text{~and~}\quad d(1,3)=100 \quad\text{~and~}\quad d(2,3)=100.
$$
The points $1$ and $2$ are close together and equidistant from point $3$. Since every three-point metric space is positive definite \cite[Prop. 2.4.15]{leinster_magnitude_2013}, its magnitude exists for all $t>0$. The black lines in Figure \ref{fig: 3point_metric_space} plot the magnitude of $X$: for small $t$ we effectively see one point, for medium $t$ we see two points, i.e., $\{1,2\}$ and $3$, and for large $t$ we see three points. 
\begin{figure}[h!]
\centering
\includegraphics[width=0.99\textwidth]{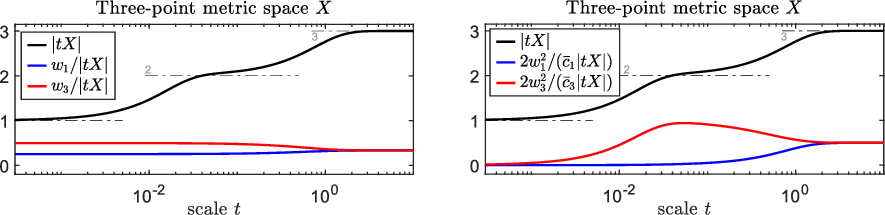}
\caption{Both panels: The magnitude $\vert tX\vert$ (black line) of the three-point metric space in Example \ref{ex: 3-point}; note the logarithmic scale for the horizontal axis. Left panel: The relative contribution $w_x/\vert tX\vert$ for points 1 and 2 (blue line) and 3 (red line). Right panel: The change contribution $2w^2_x/(\bar{c}_x\vert tX\vert)$ for points 1 and 2 (blue line) and 3 (red line).}
\label{fig: 3point_metric_space}
\end{figure}

The left and right panels in Figure \ref{fig: 3point_metric_space} compare two possible ways to measure the contribution of a point $x$ to the magnitude: $w_x/\vert X\vert$ on the left, and $\tfrac{2w^2_x}{\bar{c}_x\vert X\vert}$ on the right, inspired by Corollary \ref{cor: magnitude Y}. Overall, the latter appears to be more sensitive to the changing geometry across scales.
\end{example}
~\\
\textbf{As a third application}, we initiate the study of a new class of metric spaces; we call these spaces \emph{strongly positive definite} and they are defined by the additional positivity conditions $\mathbf{w}>0$ and $c_{ij}=-(K^\dagger)_{ij}>0$, for all $i\neq j$. We show that this class is closed under taking subspaces, and that every metric space $tX$ becomes strongly positive definite when $t\gg 0$. 

As the main application, we prove two valuative results on the magnitude of a strongly positive definite metric space $X$. These results are an attempt to reproduce results for magnitude in the spirit of the inclusion--exclusion property of the Euler characteristic: for (suitable) subspaces $A,B\subseteq \mathcal{X}$ of a topological space $\mathcal{X}$, the Euler characteristic $\chi$ satisfies
$$
\chi(A)+\chi(B) \,=\, \chi(A\cup B)+\chi(A\cap B).
$$
This fact has inspired research into valuative properties of magnitude: Leinster showed an inclusion--exclusion result for magnitude of graphs in \cite{leinster_graph_2019}, Leinster and Willerton showed in \cite{leinster_2013_asymptotic} that the magnitude of subsets of $\R^n$ satisfies inclusion--exclusion asymptotically, and Gimperlein, Goffeng and Louca further improved on these geometric results in \cite{gimperlein_magnitude_2024}. Here, we consider a variant of inclusion--exclusion; a function $f:2^X\rightarrow \R$ is called \emph{strictly submodular} if\footnote{Later in this article we will work with a different, but equivalent definition \eqref{def: submodularity}.}
$$
f(A) + f(B) \,>\, f(A\cup B) + f(A\cap B),
$$
for all $A,B\subseteq X$, and $f$ is called increasing (resp. decreasing) if it is increasing (resp. decreasing) with respect to the subset partial order. We show the following two results:
\begin{customthm}{\ref{thm: inverse submodularity}}
Let $X$ be strongly positive definite. Then the function
$$
f: Y\longmapsto 
\begin{cases}
-\vert Y\vert^{-1} \textup{,~if $Y\neq \emptyset$} \\ 
\,\,\,\,\,\,\alpha \,\,\,\quad\textup{,~if $Y=\emptyset$}
\end{cases}
$$
is increasing if $\alpha<-1$ and strictly submodular if $\alpha< -\tfrac{3}{2}$. In particular, $f$ is strictly submodular on the subspaces of $tX'$, for any $X'$ and $t\gg 0$ and $\alpha<-\tfrac{3}{2}$.
\end{customthm}
\begin{customthm}{\ref{thm: shifted submodularity}}
Let $X$ be any metric space and $t\gg 0$. Then the function
$$
f:Y \longmapsto \begin{cases}
\frac{m-\vert tY\vert}{m^2} + \frac{m-1}{m} \textup{~~~if $m:=\#Y \neq 0$}\\
\quad\quad\quad \alpha \quad\quad\quad\textup{~if $Y=\emptyset$}
\end{cases}
$$
is increasing if $\alpha<\tfrac{1}{2}$ and strictly submodular if $\alpha<-\tfrac{1}{2}$.
\end{customthm}
~\\
\textbf{Organization:} Section \ref{sec: magnitude circumradius} starts with the definition of magnitude, positive definite metric spaces, and the similarity embedding $\varphi$; this leads to the proof of Theorem \ref{thm: magnitude circumradius}, which relates the magnitude $\vert X\vert$ and circumradius $R(S)$. Section \ref{sec: asymptotics} is rather short and deals with the asymptotics of $q(tX)$, in Theorem \ref{thm: asymptotics of q}. In Section \ref{sec: matrix theory}, we develop the matrix theory of the matrices $Z,K$ and $K^\dagger$: we study their relations, their connections to magnitude and weightings and prove Theorem \ref{thm: fiedler-bapat}. This is then used to derive Theorem \ref{thm: magnitude Y general} and its corollaries, which describe the magnitude and weighting of subspaces $Y\subseteq X$ in terms of the data of $X$. Finally, Section \ref{sec: strong positive definite} deals with the new class of strongly positive definite metric spaces. We define this class, study some of its properties and show that it describes the asymptotics of $tX$ for general $X$. These results are then used to prove the submodularity results, Theorem \ref{thm: inverse submodularity} and \ref{thm: shifted submodularity}.
\\~\\
Shortly after publishing a first version of this article, the preprint \cite{asao_geometric_2025} by Asao and Gomi appeared, with a similar result as Theorem \ref{thm: magnitude circumradius}, obtained independently and with different applications.

\section{Magnitude and circumradius}\label{sec: magnitude circumradius}

\subsection{Definitions: magnitude, weighting and positive definiteness}
We start with the definition of magnitude and positive definite metric spaces; recall that the similarity matrix $Z=Z(X)$ of a finite metric space $(X,d)$ has entries defined as $z_{ij}=e^{-d(i,j)}$.
\begin{definition}[Weighting, magnitude]
Let $X$ be a finite metric space. 
\begin{itemize}
\item A vector $\mathbf{w}\in\R^X$ that satisfies $Z\mathbf{w}=\mathbf{1}$ is called a \emph{weighting on $X$}.
\item The \emph{magnitude} of $X$ is the sum $\vert X\vert =\mathbf{1}^T\mathbf{w}$, for any weighting $\mathbf{w}$ on $X$.
\end{itemize}
If the system has no solutions, we say that the magnitude and weighting do not exist. By definition, we set the magnitude of the empty metric space $X=\emptyset$ to zero. We use $n$ or $\#X$ to denote the cardinality of $X$, and we assume that $0<\#X<\infty$ unless explicitly stated otherwise. 
\end{definition}
\begin{definition}\label{def: positive definite}
A metric space $X$ is \emph{positive definite} if $Z(X)$ is positive definite.
\end{definition}

The magnitude of a positive definite metric space always exists. In fact, since the similarity matrix $Z$ is invertible in this case, there is a unique weighting $\mathbf{w}=Z^{-1}\mathbf{1}$ on $X$, and the magnitude of a positive definite metric space $X$ has the explicit form:
$$
\vert X\vert \,=\, \sum_{i,j\in X}(Z^{-1})_{ij} \,=\, \mathbf{1}^TZ^{-1}\mathbf{1}.
$$
Positive definite metric spaces play an important role in the theory of magnitude. Not only is much of the theory of magnitude easier to develop for positive definite metric spaces, as is the case in this article, but the following result due to Leinster states that any metric space eventually becomes positive definite when scaling its metric $d\mapsto t\cdot d$; we write $tX = (X,t\cdot d)$.
\begin{proposition}[{\cite[Prop. 2.4.6]{leinster_magnitude_2013}}]\label{prop: tX positive definite}
For any metric space $X$ and $t\gg 0$, $tX$ is positive definite.
\end{proposition}
The following closure result will be useful later on.
\begin{proposition}\label{prop: Y positive definite}
If $X$ is positive definite then any subspace $Y\subseteq X$ is positive definite.
\end{proposition}
\begin{proof}
Any principal submatrix $Z(Y)$ of a positive definite matrix $Z(X)$ is positive definite.
\end{proof}

To conclude, we give an important class of examples of positive definite metric spaces.
\begin{example}[Negative type metric]
A metric space $(X,d)$ has negative type if $(X,\sqrt{d})$ admits an isometric Euclidean embedding. As shown by Meckes in \cite{meckes_2012_positive}, a metric space $X$ has negative type if and only if $tX$ is positive definite for all $t>0$; this is also called stably positive definite. There are many natural examples of negative type metric spaces, for instance point sets in Euclidean, spherical or hyperbolic space. See \cite[Thm. 3.6]{meckes_2012_positive} for a list of examples.
\end{example}

\subsection{Similarity embedding of positive definite metric spaces}
In this section, we introduce an embedding of a positive definite metric space and show its relation to magnitude. This embedding makes critical use of positive definiteness. 
\begin{definition}[Similarity embedding]
A \emph{similarity embedding} of a positive definite metric space $X$ is an embedding $\varphi:X\rightarrow \R^{X-1}$ that satisfies
$$
\Vert \varphi(i)-\varphi(j)\Vert^2 \,\,=\,\, 1-e^{-d(i,j)}\,\,=\,\, 1-z_{ij}\quad\text{~for all~} i,j\in X.
$$
\end{definition}
Recall the construction from the introduction, which proves existence of these embeddings
\begin{proposition}\label{prop: similarity embedding}
The similarity embedding exists and is unique up to rigid transformations.
\end{proposition}
\begin{proof}
Let $X$ be a positive definite metric space and construct the centered similarity matrix 
$$
K \,:=\, \tfrac{1}{2}\big(I-\tfrac{\mathbf{11}^T}{n}\big)Z\big(I-\tfrac{\mathbf{11}^T}{n}\big),
$$
where $n=\# X$ is the cardinality of the metric space and $I$ is the identity matrix. Since $Z$ is positive definite, the matrix $K$ is positive semidefinite with $\ker(K)=\spn(\mathbf{1})$ (see also Proposition \ref{prop: interlacing}) and thus there exists an $(n-1)\times n$ square root $\sqrt{K}$, i.e., such that $\sqrt{K}^T\sqrt{K}=K$. The embedding $\varphi:X\ni i\mapsto(i$th column of $\sqrt{K})\in\R^{n-1}$ is a similarity embedding, since 
$$
\Vert\varphi(i)-\varphi(j)\Vert^2 \,=\, \frac{1}{2}(\mathbf{e}_i-\mathbf{e}_j)^T\sqrt{K}^T\sqrt{K}(\mathbf{e}_i-\mathbf{e}_j) \,=\, \tfrac{1}{2}(z_{ii}+z_{jj}-2z_{ij}) \,=\, 1-e^{-d(i,j)},
$$
where $\mathbf{e}_i$ is the $i$th basis vector. Since the embedding $\varphi$ is defined via all pairwise distances between the points $\varphi(X)$, it is unique up to rigid transformations, i.e., isometries of $\R^{n-1}$.
\end{proof}
\begin{remark}[Computation]
The similarity embedding $\varphi$ can be computed efficiently using standard linear algebra methods. For instance, one can compute $\sqrt{K}$ via the Cholesky decomposition or an eigenvalue decomposition, for which highly optimized algorithms exist.
\end{remark}

Proposition \ref{prop: similarity embedding} associates to each positive definite metric space $X$ a set of points $S:=\varphi(X)$ in $\R^{X-1}$. As introduced, we will show that a certain geometric invariant of this point set is related to magnitude. We first establish another important fact about the geometry of $S$.
\begin{proposition}\label{prop: S simplex}
Let $X$ be positive definite. Then $S=\varphi(X)$ are the vertices of a simplex.
\end{proposition}
\begin{proof}
This is a consequence of positive definiteness: since $Z$ has full rank, the matrix $K$ has rank $(n-1)$ and its square root, the $(n-1)\times n$ matrix $\sqrt{K}$, has full rank $n-1$. It follows that the columns of $\sqrt{K}$ and thus the points $S$ are affinely independent.
\end{proof}

\subsection{Magnitude and circumradius}
We now come to the definition of circumradius. This is the geometric invariant of the point set $S$ that is closely related to magnitude; we note that, by construction, the matrix $(\mathbf{11}^T-Z)$ is the matrix of squared distances between pairs of points in $S$, and we may thus write $Z(S)$.
\begin{theorem}\label{thm: definition circumradius}
The following numbers are equal for the vertices $S$ of a simplex:
\begin{enumerate}
\item The radius $r$ of the unique sphere in $\R^{S-1}$ that goes through $S$;
\item The unique value $r$ for which $(\mathbf{11}^T-Z(S))\mathbf{x}=2r^2\cdot\mathbf{1}$ with $\mathbf{x}^T\mathbf{1}=1$ has a solution.
\end{enumerate}
This common value is called the \emph{circumradius} of $S$ and is denoted by $R(S)$.
\end{theorem}
\begin{proof}
From Euclidean geometry (or algebraic geometry), we know there is a unique sphere in $\R^{S-1}$ that goes through the vertices $S$ of a simplex. This is the \emph{circumsphere} of the simplex.

(1.~$\Leftrightarrow$ 2.) The circumsphere has center $\mathbf{c}$ (= the \emph{circumcenter}), which satisfies the equations:
$$
\Vert \mathbf{c}-\mathbf{s}\Vert^2\,=r^2 \quad\text{~for each $\mathbf{s}\in S$,}
$$
where $r$ is the radius of the circumsphere. If we write $\mathbf{p}$ for the barycentric coordinate of the circumcenter, i.e., the unique vector such that $\mathbf{1}^T\mathbf{p}=1$ and $\sqrt{K}\mathbf{p}=\mathbf{c}$, and use the fact that for each $i\in X$, the vector $\mathbf{e}_i$ is the barycentric coordinate of the $i$th point of $S$, then we can rewrite
\begin{align*}
&(\mathbf{p}-\mathbf{e}_i)^TK(\mathbf{p}-\mathbf{e}_i) \,= r^2 \quad\text{~~~for each $i\in X$}
\\
\iff&\tfrac{1}{2}(\mathbf{p}-\mathbf{e}_i)^TZ(\mathbf{p}-\mathbf{e}_i)\, = r^2 \quad\,\text{~for each $i\in X$}\quad\text{(since $\mathbf{1}^T(\mathbf{p}-\mathbf{e}_i)=0$)}
\\
\iff&\mathbf{e}_i^TZ\mathbf{p} = -r^2 + \tfrac{1}{2}\mathbf{p}^TZ\mathbf{p}+\tfrac{1}{2} \quad\text{for each $i\in X$}
\\
\iff&Z\mathbf{p} = (\tfrac{1}{2} - r^2 + \tfrac{1}{2}\mathbf{p}^TZ\mathbf{p})\cdot\mathbf{1}.
\\
\iff&Z\mathbf{p} = (1-2r^2)\cdot\mathbf{1} \hphantom{putting some distance betwee}\text{~(since $\mathbf{1}^T\mathbf{p}=1$)}
\\
\iff&(\mathbf{11}^T-Z)\mathbf{p}=2r^2\cdot\mathbf{1}.
\end{align*}
Thus the radius $r$ of the circumsphere satisfies the required equation (2.). Uniqueness of $r$ in (2.) follows from uniqueness of the sphere that contains all vertices $S$ of a simplex, or it can be derived from invertibility of $Z$ as implied by positive definiteness.
\end{proof}
\begin{remark}[Scaled embeddings]
Any rescaling $\alpha\cdot\varphi$ for $\alpha>0$ of the similarity embedding still embeds the metric space $X$ as the vertices of a simplex, with circumradius $\alpha R(S)$. Different choices than $\alpha=1$ may be useful in different applications, but some of the nice theory may be lost, for instance, when choosing a scaling that depends on $X$. For instance, for $\alpha^2=\vert X\vert$ or $\alpha^2=\# X$, the subspace characterization in Theorem \ref{thm: magnitude circumradius Y} below no longer holds. In a sense, $X$-dependent scalings are not functorial in the same way as independent scalings are.
\end{remark}

We now arrive at the first main result from the introduction, and its proof.
\begin{theorem}\label{thm: magnitude circumradius}
Let $X$ be a positive definite metric space with similarity embedding $S$. Then
\begin{equation}\label{eq: main result}
\vert X\vert \,=\, \frac{1}{1-2R(S)^2}.
\end{equation}
\end{theorem}
\begin{proof}
This follows immediately from the third expression for the circumradius in Theorem \ref{thm: definition circumradius}. From $Z\mathbf{x} = (1-2R(S)^2)\mathbf{1}$ with $\mathbf{x}^T\mathbf{1}=1$, we find that $\mathbf{w}=\mathbf{x}/(1-2R(S)^2)$ is the weighting on $X$. Since $\mathbf{x}$ is normalized, the claimed expression for the magnitude $\vert X\vert$ follows.
\end{proof}

\begin{remark}[Equilibrium equations]\label{rem: equilibrium equations}
With the proofs of Theorem \ref{thm: definition circumradius} and Theorem \ref{thm: magnitude circumradius} in place, we can now appreciate where the relation between magnitude and circumradius comes from: both quantities are defined via an `equilibrium equation'. In the case of magnitude, the equilibrium equation $Z\mathbf{w}=\mathbf{1}$ is part of the definition which Leinster arrived at by generalizing the Euler characteristic. In the geometric setting, the equilibrium equation $(\mathbf{11}^T-Z)\mathbf{p}=2r^2\mathbf{1}$ encodes the fact that the circumcenter is equidistant to all vertices of a simplex.
\end{remark}
\begin{corollary}
Let $X$ be a positive definite metric space. Then $\vert X\vert>1$.
\end{corollary}
\begin{proof}
This follows from $\vert X\vert = 1+2\vert X\vert R^2(S)$ with $\vert X\vert>0$ and $R(S)>0$.
\end{proof}

We can in fact turn Theorem \ref{thm: magnitude circumradius} around and use it to characterize the similarity embedding. 
\begin{theorem}\label{thm: magnitude circumradius Y}
Let $X$ be a positive definite metric space. The similarity embedding $\varphi$ is the unique embedding of $X$ into $\R^{X-1}$, up to rigid transformation, such that
$$
\vert Y\vert = \frac{1}{1-2R(\varphi(Y))^2}\textup{~for all nonempty $Y\subseteq X$.}
$$
\end{theorem}
\begin{proof}
We recall from Example \ref{ex: 2-point} that the magnitude of a two-point metric space is $2/(1+e^{-d})$. Furthermore, the circumradius of two points at Euclidean distance $d'$ is $d'/2$. 

For every two-point subspace $Y_{ij}=\{i,j\}\subseteq X$, the requirement on the embedding $\varphi$ gives 
$$
\frac{2}{1+e^{-d(i,j)}} = \vert Y_{ij}\vert = \frac{1}{1-2R(\varphi(Y_{i,j}))^2} = \frac{2}{2-\Vert\varphi(i)-\varphi(j)\Vert^2} \iff \Vert\varphi(i)-\varphi(j)\Vert^2=1-e^{-d(i,j)}.
$$
The latter equation shows that any embedding with the prescribed conditions must be the similarity embedding, which is unique. It remains to prove existence, i.e., that the similarity embedding $\varphi$ satisfies the stated relation between $\vert Y\vert$ and $R(\varphi(Y))$ for subspaces $Y$ of any cardinality. This follows because the restriction $\varphi\vert_{Y}$ of the similarity embedding of $X$ onto any subspace $Y$ is the similarity embedding of $Y$, and $R(\varphi(Y))=R(\varphi\vert_Y(Y))$. The expression for $\vert Y\vert$ in terms of $R(\varphi(Y))$ is then simply Theorem \ref{thm: magnitude circumradius}.
\end{proof}

Another way to think about Theorem \ref{thm: magnitude circumradius} is that it expresses functoriality of the similarity embedding with respect to the subspace relation. This is also reflected in the matrix identity \eqref{eq: fiedler-bapat} and its implication, e.g., Theorem \ref{thm: magnitude Y general}.

To end this section, we compute the similarity embedding, magnitude and circumradius of a three-point metric space in a numerical example.
\begin{example}
Consider the three-point metric space $X=\{1,2,3\}$ with metric $d(1,2)=\ln(2)\approx 0.69$ and $d(1,3)=d(2,3)=\ln(10)\approx 4.61$. Its similarity matrix, centered similarity matrix and a similarity embedding are:
$$
Z = \begin{pmatrix}1&0.5 & 0.1\\0.5&1&0.1\\0.1&0.1&1\end{pmatrix}\quad K = \tfrac{1}{9}\begin{pmatrix}1.9&-0.35&-1.55\\-0.35&1.9&-1.55\\-1.55&-1.55&3.1\end{pmatrix} \quad \begin{aligned}
\sqrt{K}&=\Big(\varphi(1)~\varphi(2)~\varphi(3)\Big) \\&= \begin{pmatrix}\sqrt{2}/4 & -\sqrt{2}/4 & 0\\
-0.2934&-0.2934&0.5869\end{pmatrix}
\end{aligned}
$$
\end{example}

\section{Magnitude asymptotics}\label{sec: asymptotics}
As explained in the introduction, much of the theory of magnitude is concerned with studying how the magnitude $\vert tX\vert $ behaves as a function of the scale $t>0$. Our goal here is to refine the following asymptotic formula, due to Leinster and Willerton in \cite{leinster_2013_asymptotic}:
$$
\vert tX\vert= n-q(tX)\text{~with~}q(tX)\rightarrow 0\text{~as $t\rightarrow\infty$},
$$
where $n$ is the cardinality of the metric space. As a direct application of Theorem \ref{thm: magnitude circumradius}, we give an asymptotic expression for the remainder $q(tX)$ in this expression; we write $S_t:=\varphi(tX)$.
\begin{theorem}\label{thm: asymptotics of q}
Let $X$ be a metric space on $n$ points. Then we have the asymptotic equivalence
$$n-\vert tX\vert=q(tX) \,\sim\, n^2\left( \frac{n-1}{n} - 2R(S_t)^2\right).$$
\end{theorem}
\begin{proof}
As $t\rightarrow \infty$, $tX$ is eventually positive definite (Proposition \ref{prop: tX positive definite}) and its similarity matrix $Z(tX)$ converges to the identity matrix $I$. As a result, the simplex with vertices $S_t$ converges to the regular simplex with all edge lengths equal to $\sqrt{1-z_{ij}}=1$. From the equilibrium equation, one computes that $R(S_t)^2$ converges to $\frac{n-1}{2n}$, the squared circumradius of the regular simplex.

Next, based on Theorem \ref{thm: magnitude circumradius}, we can write the remainder $q(tX)=n-\vert tX\vert$ as follows:
$$
\vert tX \vert = \frac{1}{1-2R(S_t)^2} = \frac{n}{1+n\big(\frac{n-1}{n}-2R(S_t)^2\big)}\Rightarrow q(tX) = n\left(1-\frac{1}{1+n\big(\frac{n-1}{n}-2R(S_t)^2\big)}\right)
$$
Using the asymptotic equivalence $1-\frac{1}{1+x}\sim x$ for $x\rightarrow 0$ and the earlier result that $\big(\frac{n-1}{n}-2R(S_t)^2\big)\rightarrow 0$ as $t\rightarrow 0$, we then find the claimed asymptotic equivalence.
\end{proof}

Theorem \ref{thm: asymptotics of q} states that the magnitude of a metric space $tX$ converges to the magnitude of the discrete metric space $(X,\infty)$ in the same way as the circumradius of the simplex $S_t$ converges to the circumradius of the regular simplex. We continue with magnitude asymptotics in Section \ref{sec: strong positive definite}, where we will further refine our understanding of the asymptotics by looking more closely at some implications of the convergence of $S_t$ to the regular simplex. 

\section{Matrix theory \& magnitude of subspaces}\label{sec: matrix theory}
In this section, we look more closely at the matrices $Z$ and $K$ that appeared in Section \ref{sec: magnitude circumradius}, and introduce and study the matrix $K^\dagger$. This section is partly an intermezzo, where we develop the matrix theory of magnitude for its own sake, and partly a setup for Section \ref{sec: strong positive definite}.


\subsection{Matrix theory of $Z,K$ and $K^\dagger$}
The \emph{Moore--Penrose pseudoinverse} of a matrix $M$ is the unique matrix $M^\dagger$ that satisfies
$$
\quad MM^\dagger M = M\quad\text{~and~}\quad M^\dagger MM^\dagger = M^\dagger\quad\text{~and~}\quad MM^\dagger, M^\dagger M\text{~are symmetric}.
$$
We point out two consequences of this definition: (i) if $M$ is real and symmetric, then so is $M^\dagger$ and (ii) the matrices $MM^\dagger$ and $M^\dagger M$ are orthogonal projection matrices onto $\ker(M)=\ker(M^\dagger)$ (see, e.g., \cite[\S A.3]{devriendt_thesis_2022}). We recall the definition of matrices $K$ and $K^\dagger$:
\begin{definition}
Let $X$ be a finite metric space with similarity matrix $Z$. Then
\begin{itemize}
\item The \emph{centered similarity matrix} is $K := \tfrac{1}{2}\big(I-\tfrac{\mathbf{11}^T}{n}\big)Z\big(I-\tfrac{\mathbf{11}^T}{n}\big)$.
\item The \emph{pseudoinverse centered similarity matrix} is $K^\dagger$.
\end{itemize}
\end{definition}
By their construction, we immediately find the following properties of $K$ and $K^\dagger$:
\begin{proposition}
$K$ and $K^\dagger$ are real, symmetric matrices with $\ker(K)=\ker(K^\dagger)\supseteq\spn(\mathbf{1})$. 
\end{proposition}
\begin{proposition}[Interlacing]\label{prop: interlacing}
Let $\lambda_1\geq \lambda_2\geq\dots\geq \lambda_n$ be the real eigenvalues of $Z$, and $\mu_1\geq\mu_2\geq\dots\geq \mu_{n-1}$ the real eigenvalues of $K$, with one zero eigenvalue omitted\footnote{By construction we know that $K$ has at least one eigenvalue. The eigenvalues $\mu_1,\dots,\mu_{n-1}$ are all other eigenvalues except for this one zero eigenvalue; there may be other zero eigenvalues.}. Then 
$$
\lambda_1\geq 2\mu_1 \geq \lambda_2 \geq 2\mu_2 \geq \dots \geq \lambda_{n-1}\geq 2\mu_{n-1}\geq \lambda_{n},
$$
with equality throughout if $Z\mathbf{1}=\vert X\vert^{-1}\mathbf{1}$. We have $\rank(K)=\rank(K^\dagger)=\rank(Z)-1$ and if $Z$ is positive definite, then $K$ and $K^\dagger$ are positive semidefinite with $\ker(K)=\ker(K^\dagger)=\spn(\mathbf{1})$.
\end{proposition}
\begin{proof}
The eigenvalue inequalities are precisely the interlacing theorem for matrix projections \cite[Thm. 2.1]{haemers_1995_interlacing} applied to the matrices $2K$ and $Z$. If $Z\mathbf{1}=\vert X\vert^{-1}\cdot\mathbf{1}$ then $\vert X\vert^{-1}$ is an eigenvalue of $Z$ with (normalized) eigenvector $\mathbf{1}/\sqrt{n}$; let's say that it is the $j$th eigenvalue. Then $Z$ has an eigendecomposition $Z=\sum_{i\neq j}\lambda_i\mathbf{xx}^T+\vert X\vert^{-1}\tfrac{\mathbf{11}^T}{n}$ and thus by its construction $K$ has eigendecomposition $K=\sum_{i\neq j}\tfrac{1}{2}\lambda_i\mathbf{xx}^T$, which proves that $2\mu_i=\lambda_i$ for all $i\neq j$. 

The implication ($Z$ positive definite)$\implies$($K$ and $K^\dagger$ positive semidefinite) follows from the interlacing inequalities. The rank statement holds because $2K$ arises from $Z$ through a projection onto $\spn(\mathbf{1})^\perp$, where furthermore $\ker(Z)\not\supseteq\spn(\mathbf{1})$ since $Z$ is a nonnegative matrix with unit-diagonal. By the rank statement, $K$ and $K^\dagger$ have rank $n-1$ when $Z$ is positive definite, and thus $\ker(K)=\ker(K^\dagger)=\spn(\mathbf{1})$. This completes the proof.
\end{proof}

One example where $Z\mathbf{1}=\vert X\vert^{-1}\mathbf{1}$ holds is when the isometry group of $X$ acts transitively on its points; these are called \emph{homogeneous spaces} \cite[Def. 2.1.4]{leinster_magnitude_2013}. We continue by showing some relations between $Z,K$ and $K^\dagger$. We start with a useful lemma:
\begin{lemma}
Let $Z$ be invertible with nonzero magnitude. Then
\begin{equation}\label{eq: ZK}
ZK^\dagger = 2I - 2\frac{\mathbf{1w}^T}{\vert X\vert}.
\end{equation}
\end{lemma}
\begin{proof}
We start with the definition of $K$ in terms of $Z$ and use $KK^\dagger= \big(I-\tfrac{\mathbf{11}^T}{n}\big)$:
$$
K=\tfrac{1}{2} \big(I-\tfrac{\mathbf{11}^T}{n} \big)Z\big(I-\tfrac{\mathbf{11}^T}{n}\big) 
\Rightarrow KK^\dagger = \tfrac{1}{2} \big(I-\tfrac{\mathbf{11}^T}{n}\big) ZK^\dagger \Leftrightarrow \big(I-\tfrac{\mathbf{11}^T}{n}\big) = \tfrac{1}{2} \big(I-\tfrac{\mathbf{11}^T}{n}\big)ZK^\dagger.
$$
The third equation implies that $ZK^\dagger$ must be of the form $ZK^\dagger = 2I - 2\mathbf{1y}^T$ for some $\mathbf{y}\in\R^n$. Using the additional fact that $\mathbf{w}^TZK^\dagger = 0$, we find that $\mathbf{y}=\mathbf{w}/\vert X\vert$ as required. 
\end{proof}

We now show some relations between $Z, K$ and $Z^{-1},K^\dagger$.
\begin{proposition}
The similarity and centered similarity matrices satisfy:
$$
2K-Z = \bgamma\mathbf{1}^T+\mathbf{1}\bgamma^T,\quad\textup{~where~}\bgamma \,:=\, -\tfrac{1}{n}Z\mathbf{1} + \tfrac{1}{2n^2}\mathbf{1}^TZ\mathbf{1}\cdot\mathbf{1} \,=\, \diag(K) - \tfrac{1}{2}\mathbf{1}.
$$
\end{proposition}
\begin{proof}
We note that $Z$ may be singular. Starting from the definition of $K$, we find
$$
2K-Z = -\tfrac{1}{n}Z\mathbf{11}^T - \tfrac{1}{n}\mathbf{11}^TZ + \tfrac{1}{n^2}\mathbf{11}^TZ\mathbf{11}^T,
$$
which proves the expression for $2K-Z$. The second expression of $\bgamma$ follows from $\diag(Z)=\mathbf{1}$.
\end{proof}
\begin{proposition}
Let $Z$ be invertible with nonzero magnitude. Then
$$
2Z^{-1} - K^\dagger = 2\frac{\mathbf{ww}^T}{\vert X\vert}, \quad\textup{~where~}\begin{cases}\tfrac{\mathbf{w}}{\vert X\vert} = \tfrac{1}{2}K^\dagger\bgamma + \tfrac{1}{n}\mathbf{1};\\
\vert X\vert^{-1} = -\tfrac{1}{2}\bgamma^TK^\dagger\bgamma - \tfrac{2}{n}\mathbf{1}^T\bgamma.
\end{cases}
$$
\end{proposition}
\begin{proof}
Right-multiplying equation \eqref{eq: ZK} by $Z^{-1}$ proves the expression for $2Z^{-1}-K^\dagger$. We now prove the expression for $\mathbf{w}/\vert X\vert$ in terms of $K^\dagger$ and $\bgamma$.
\begin{align*}
K^\dagger = 2Z^{-1}-2\tfrac{\mathbf{ww}^T}{\vert X\vert} \Rightarrow K^\dagger\bgamma &= \big(2Z^{-1}-2\tfrac{\mathbf{ww}^T}{\vert X\vert}\big)\big(-\tfrac{1}{n}Z\mathbf{1}+\tfrac{1}{2n^2}\mathbf{1}^TZ\mathbf{1}\cdot\mathbf{1}\big)
\\
&= -\tfrac{2}{n}\mathbf{1} + \tfrac{2}{\vert X\vert}\mathbf{w} \Rightarrow \tfrac{1}{2}K^\dagger\bgamma = -\tfrac{1}{n}\mathbf{1}+\tfrac{\mathbf{w}}{\vert X\vert}.
\end{align*}
Finally, we prove the expression for $\vert X\vert^{-1}$ in terms of $K^\dagger$ and $\bgamma$.
\begin{align*}
\tfrac{1}{2}K^\dagger\bgamma + \tfrac{1}{n}\mathbf{1} = \tfrac{\mathbf{w}}{\vert X\vert} \Rightarrow \tfrac{1}{2}\bgamma^TK^\dagger\bgamma + \tfrac{1}{n}\bgamma^T\mathbf{1} &= \tfrac{\mathbf{w}^T\bgamma}{\vert X\vert}
\\
&= \tfrac{\mathbf{w}^T}{\vert X\vert}\big(-\tfrac{1}{n}Z\mathbf{1} + \tfrac{1}{2n^2}\mathbf{1}^TZ\mathbf{1}\cdot\mathbf{1}\big)
\\
&= -\tfrac{1}{\vert X\vert} + \tfrac{1}{2n^2}\mathbf{1}^T Z\mathbf{1}
\\
&= - \tfrac{1}{\vert X\vert} - \tfrac{1}{n}\bgamma^T\mathbf{1},
\end{align*}
which leads to the claimed expression for $\vert X\vert^{-1}$ and completes the proof.
\end{proof}
\\~\\
We focus on one further `object' that turns out to play an important role in the matrix theory and the asymptotics of magnitude: the off-diagonal entries of matrix $K^\dagger$. Later, in Section \ref{sec: strong positive definite}, we will show that large scales $t\gg 0$ are marked by negativity of these entries, which in turn leads to additional algebraic and geometric properties. We introduce the notation
$$
c_{ij} := -(K^\dagger)_{ij} = 2\tfrac{w_iw_j}{\vert X\vert} - 2(Z^{-1})_{ij}\text{,~for $i\neq j\in X$} \quad\text{~and~}\quad \bar{c}_i := (K^\dagger)_{ii} = \sum_{j\in X}c_{ij}.
$$
The relation between $\bar{c_i}$ and $\sum_{j}c_{ij}$ follows because the rows of $K^\dagger$ sum to zero. As a notational convention, when summing over $j\in X$ or $i,j\in X$ with $c_{ij}$ in the sum, we consider each pair only once (i.e., with $(i,j)=(j,i)$); think of $c$ as a function on unordered pairs of distinct points. 

We give expressions for magnitude and weightings in terms of $c$.
\begin{proposition}\label{prop: magnitude and c}
Let $Z$ be invertible with nonzero magnitude. Then
$$
\frac{w_i}{\vert X\vert} = 1 - \frac{1}{2}\sum_{j\in X}c_{ij}(1 - z_{ij})\quad\textup{~and~}\quad \vert X\vert^{-1} = 1 - \sum_{i,j\in X} c_{ij}(z_{ix}-z_{jx})^2\textup{~for any $x\in X$}.
$$
\end{proposition}
\begin{proof}
The expression for $w_i$ follows by taking the $i$th diagonal entry of equation \eqref{eq: ZK}. For the expression of $\vert X\vert^{-1}$ we start by the observation that, due to $\ker(K^\dagger)=\spn(\mathbf{1})$, we can write
$$
K^\dagger = \sum_{i,j\in X}c_{ij}(\mathbf{e}_i-\mathbf{e}_j)(\mathbf{e}_i-\mathbf{e}_j)^T \,\Rightarrow\, \tfrac{1}{2}\mathbf{y}^TK^\dagger\mathbf{y} = \sum_{i,j\in X}c_{ij}(y_i-y_j)^2 \text{~for any $\mathbf{y}\in\R^X$.}
$$
Right-multiplying equation \eqref{eq: ZK} by $Z$ and looking at the $x$th diagonal entry, we then find
$$
(ZK^\dagger Z)_{xx} = 2 - 2\vert X\vert^{-1} \Leftrightarrow \vert X\vert^{-1} = 1-\sum_{i,j \in X}c_{ij}(z_{ix}-z_{jx})^2.
$$
We note that a factor of $\tfrac{1}{2}$ is incorporated in the sum because of our notational convention.
\end{proof}

\begin{corollary}
Let $Z$ be invertible with nonzero magnitude. Then $\sum_{i,j\in X}c_{ij}(1 - z_{ij})=n-1$.
\end{corollary}
\begin{proof}
This follows by summing $w_i/\vert X\vert$ in Proposition \ref{prop: magnitude and c}, and $\mathbf{1}^T\mathbf{w}=\vert X\vert$.
\end{proof}
\begin{remark}[Foster's theorem] 
The direct analogue of this corollary is called Foster's Theorem in other settings. One hint at the depth of this theorem is that the integer on the righthandside is known to count, depending on the situation, a dimension, a degree or an Euler characteristic.
\end{remark}

\begin{corollary}
Let $Z$ be invertible with nonzero magnitude. Then
$$
\vert X\vert^{-1} \,\,=\,\, \tfrac{1}{n}\tr\Big(\big(I-\tfrac{1}{2}ZK^\dagger Z\big)\Big).
$$
\end{corollary}
\begin{proof}
This follows by summing the expression for $\vert X\vert^{-1}$ in Proposition \ref{prop: magnitude and c} over all $x\in X$.
\end{proof} 

To end the section, we prove a matrix identity that relates the similarity and pseudoinverse centered similarity matrices of a metric space. This identity summarizes several of the expressions above, and it will be the basis to derive explicit relations between different subspaces. 
\begin{theorem}\label{thm: fiedler-bapat}
Let $X$ be a metric space with invertible $Z$ and nonzero magnitude. Then
\begin{equation}\label{eq: fiedler-bapat}
\begin{pmatrix}0 & \mathbf{1}^T\\\mathbf{1}& Z \end{pmatrix}^{-1}
\,=\,\,\,
\begin{pmatrix}-\vert X\vert^{-1} & \mathbf{w}^T/\vert X\vert \\ \mathbf{w}/\vert X\vert& \tfrac{1}{2} K^\dagger\end{pmatrix}.
\end{equation}
\end{theorem}
\begin{proof}
We verify that multiplying the two sides of equation \eqref{eq: fiedler-bapat} gives the identity matrix:
$$
\begin{pmatrix}
\mathbf{1^Tw}/\vert X\vert 
& \tfrac{1}{2}\mathbf{1}^TK^\dagger
\\
\vert X\vert^{-1}\cdot(Z\mathbf{w}-\mathbf{1}) 
& 
\mathbf{1w^T}/\vert X\vert +\tfrac{1}{2} ZK^\dagger
\end{pmatrix} \,\,=\,\, \begin{pmatrix}
1&0\\0&I
\end{pmatrix}.
$$
The top-left equation holds by definition of magnitude. The top-right equation holds because $\ker(K^\dagger)=\spn(\mathbf{1})$. The bottom-left equation holds because $Z\mathbf{w}=\mathbf{1}$. The bottom-right equation holds by equation \eqref{eq: ZK}.
\end{proof}
\begin{remark}[Cayley--Menger matrix]
Theorem \ref{thm: fiedler-bapat} mimics a matrix identity in simplex geometry and algebraic graph theory; we called this the \emph{Fiedler--Bapat identity} in \cite[\S3]{devriendt_thesis_2022}. In that context, the matrix $\left(\begin{smallmatrix}0& \mathbf{1}^T\\\mathbf{1}&D\end{smallmatrix}\right)$ is called the \emph{Cayley--Menger matrix}, where $D$ contains the squared pairwise distances between a set of points. This matrix is an important algebraic object in geometry; for instance, its determinant encodes affine (in-)dependence of the set of points.
\end{remark}
We note the following rational expression for $\vert X\vert$ in variables $(z_{ij})$ that follows from \eqref{eq: fiedler-bapat}:
\begin{proposition}\label{prop: magnitude determinant}
Let $X$ be a positive definite metric space. Then
$$
\vert X\vert \,=\, -\det\begin{pmatrix}0&\mathbf{1}^T\\\mathbf{1}&Z\end{pmatrix}/\det (Z).
$$
\end{proposition}

\subsection{Magnitude and weightings for subspaces}
We now apply Theorem \ref{thm: fiedler-bapat} to study the magnitude of subspaces of a metric space; we start by fixing some notation. We write $Y\subseteq X$ to denote the metric subspace $(Y,d\vert_Y)\subseteq (X,d)$, and write $d$ for the metric $d\vert_Y$; unless stated otherwise, we assume $Y$ nonempty. We write $Y^c=X\setminus Y$ and use the following subscript notation for submatrices:
$$
Z = \begin{pmatrix}
Z_{YY} & Z_{YY^c}\\
Z_{Y^c Y} & Z_{Y^cY^c}
\end{pmatrix}\quad\text{~and~}\quad\mathbf{w} = \begin{pmatrix}\mathbf{w}_Y\\\mathbf{w}_{Y^c}\end{pmatrix}.
$$
Since the metric on $Y$ is simply the restriction of $d$, it follows that the similarity matrix $Z(Y)$ is the principal submatrix $Z_{YY}$ of $Z(X)$. We recall the definition of Schur complements:
\begin{definition}[Schur complement]
Let $M=\left(\begin{smallmatrix}M_{YY}&M_{YY^c}\\M_{Y^cY}&M_{Y^cY^c}\end{smallmatrix}\right)$ be a matrix such that $M_{Y^cY^c}$ is invertible. The \emph{Schur complement} of $M$ with respect to $Y^c$ is the matrix\footnote{This notation is not standard; we use it to keep the focus on the underlying metric space $X$ and its subspaces.} $M/Y^c$, with rows and columns indexed by $Y$, defined by
$$
M/Y^c := M_{YY} - M_{YY^c}(M_{Y^cY^c})^{-1}M_{Y^cY}.
$$
\end{definition}

\begin{proposition}[{\cite{haynsworth_schur_1968}}]\label{prop: schur complement}
Let $M$ be a matrix such that $M/Y^c$ is well-defined. Then
\begin{itemize}
\item If $M$ is invertible, then $(M^{-1})_{YY} = (M/Y^c)^{-1}$.
\item $M/Y^c= ((M/x_1)/\dots)/x_\ell$ with entries $Y^c=\{x_1,\dots,x_\ell\}$ in any order.
\end{itemize} 
\end{proposition}

We now arrive at the main result of this section.

\begin{theorem}\label{thm: magnitude Y general}
Let $X$ be a positive definite metric space. Then for any $Y\subseteq X$ we have
$$
\vert Y\vert = \vert X\vert \left(1 + 2\frac{\mathbf{w}_{Y^c}^T(K^\dagger_{Y^cY^c})^{-1}\mathbf{w}_{Y_c}}{\vert X\vert}\right)^{-1}
\quad\textup{~and~}\quad 
\frac{\mathbf{w'}}{\vert Y\vert} = \frac{\mathbf{w}_Y}{\vert X\vert}  - K^\dagger_{Y Y^c}(K^\dagger_{Y^cY^c})^{-1}\frac{\mathbf{w}_{Y^c}}{\vert X\vert},
$$
where $\mathbf{w}'$ is the weighting on $Y$. Furthermore, $K^\dagger(Y)=K^\dagger(X)/Y^c$.
\end{theorem}
\begin{proof}
Starting from Theorem \ref{thm: fiedler-bapat} for $X$, combined with the expression of inverse matrix submatrices in terms of Schur complements (Proposition \ref{prop: schur complement}), we find
\begin{align*}
\begin{pmatrix}
0 &\mathbf{1}^T\\
\mathbf{1}&Z_{YY}
\end{pmatrix}^{-1} &=
\begin{pmatrix}
-\vert X\vert^{-1} & \mathbf{w}_{Y}^T/\vert X\vert & \mathbf{w}_{Y^c}^T/\vert X\vert\\
\mathbf{w}_Y/\vert X\vert & \tfrac{1}{2}K^\dagger_{YY} & \tfrac{1}{2} K^\dagger_{YY^c}\\
\mathbf{w}_{Y^c}/\vert X\vert & \tfrac{1}{2} K^\dagger_{Y^c Y} & \tfrac{1}{2} K^\dagger_{Y^cY^c}
\end{pmatrix}\,\Big/\, Y^c.
\\
&= 
\begin{pmatrix}
-\vert X\vert^{-1} & \mathbf{w}_{Y}^T/\vert X\vert\\
\mathbf{w}_Y/\vert X\vert & \tfrac{1}{2}K^\dagger_{YY}\\
\end{pmatrix} - 
2\begin{pmatrix}
\mathbf{w}_{Y^c}^T/\vert X\vert
\\
\tfrac{1}{2}K^\dagger_{YY^c}
\end{pmatrix}
(K^\dagger_{Y^cY^c})^{-1}
\begin{pmatrix}
\mathbf{w}_{Y^c}/\vert X\vert & \tfrac{1}{2}K^\dagger_{Y^cY}
\end{pmatrix}.
\end{align*}
Next, Theorem \ref{thm: fiedler-bapat} for $Y$ states that
$$
\begin{pmatrix}
0 &\mathbf{1}^T\\
\mathbf{1}&Z(Y)
\end{pmatrix}^{-1}=
\begin{pmatrix}
-\vert Y\vert^{-1} & \mathbf{w}'^T/\vert Y\vert\\
\mathbf{w}'/\vert Y\vert & \tfrac{1}{2}K(Y)^\dagger
\end{pmatrix}.
$$
The theorem then follows by equating these two expressions, using the equality $Z(Y)=Z_{YY}$. 
\end{proof}
\begin{corollary}
Let $X$ be positive definite. Then $\vert Y\vert\leq \vert X\vert$ for all $Y\subset X$, with equality if and only the weighting $w$ on $X$ satisfies $w_y=0$ for all $y\in Y$.
\end{corollary}
\begin{proof}
We first establish that the matrix $K^\dagger_{Y^cY^c}$, and thus its inverse, is positive definite:
$$
\mathbf{x}^T K^\dagger_{Y^cY^c}\mathbf{x} = \begin{pmatrix}\mathbf{x}\\ \mathbf{0}_Y\end{pmatrix}^T\left( I-\frac{\mathbf{11}^T}{n}\right)K^\dagger \left( I-\frac{\mathbf{11}^T}{n}\right)\begin{pmatrix}\mathbf{x}\\ \mathbf{0}_Y\end{pmatrix} >0 \text{~for all nonzero $\mathbf{x}\in\R^{Y^c}$},
$$
which follows since $K^\dagger$ is positive definite on $\spn(\mathbf{1})^\perp$. Thus $K^\dagger_{Y^cY^c}$ is positive definite. The corollary then follows from the expression for $\vert Y\vert$ in Theorem \ref{thm: magnitude Y general}, and $\vert X\vert>0$.
\end{proof}

The special case of the subspace $Y=X\backslash\{x\}$ for some $x\in X$ is particularly useful:
\begin{corollary}\label{cor: magnitude Y}
Let $X$ be a positive definite metric space. Then for any $x\in X$ we have
$$
\vert X \backslash\{x\} \vert \,=\, \vert X\vert\left(1 + \frac{2w_x^2}{\bar{c}_x \vert X\vert}\right)^{-1} \quad \textup{~and~}\quad \frac{w'_i}{\vert X\backslash\{x\}\vert} \,=\, \frac{w_i}{\vert X\vert} + \frac{c_{ix}}{\bar{c}_{x}} \frac{w_x}{\vert X\vert},
$$
for any $i\neq x$, and where $\mathbf{w}'$ is the weighting on $X\backslash\{x\}$.
\end{corollary}
For the entries $c$ of the pseudoinverse centered similarity matrix $K^\dagger$, we find:
\begin{corollary}\label{cor: c Y}
Let $X$ be a positive definite metric space. Then for any $x\in X$ we have
$$
c_{ij}' \,=\, c_{ij} + \frac{c_{ix}c_{jx}}{\bar{c}_x} \quad\textup{~and~}\quad \bar{c}_i' \,=\, \bar{c}_i - \frac{c_{ix}^2}{\bar{c}_x},
$$
for all $i,j\neq x$ and where $c'$ is defined on $X\backslash\{x\}$, i.e., via the entries of $K(X\backslash\{x\})^\dagger$. 
\end{corollary}

\section{Strongly positive definite metric spaces}\label{sec: strong positive definite}
In this section, we return to the asymptotics of $\vert tX\vert$. Our main contribution is to identify a new property of metric spaces, which we call ``strongly positive definite", that is satisfied by every metric space $tX$ for $t\gg0$. This new regime refines our understanding of the asymptotics of $tX$.

\subsection{Definition and first properties}
We start by defining the new class of metric spaces; recall that $c_{ij}=-(K^\dagger)_{ij}=2\tfrac{w_iw_j}{\vert X\vert} - 2(Z^{-1})_{ij}$.
\begin{definition}[Strongly positive definite]\label{def: strongly positive definite}
A finite metric space $X$ is called \emph{strongly positive definite} if it is positive definite with $c>0$ and $\mathbf{w}>0$.
\end{definition}
\begin{question}
What is the right generalization of Definition \ref{def: strongly positive definite} when $X$ is not finite?
\end{question}

Strongly positive definite metric spaces are related to other known objects and properties. We summarize some of these equivalences in the two propositions below.
\begin{proposition}
The following are equivalent for a positive definite metric space $X$:
\begin{enumerate}
\item The weighting is nonnegative $\mathbf{w}\geq 0$ (resp. positive $\mathbf{w}>0$);
\item $S(X)$ are the vertices of a simplex that contains (resp. strictly contains) its circumcenter.
\end{enumerate}
\end{proposition}
\begin{proof}
In the proof of Theorem \ref{thm: magnitude circumradius}, we found that the normalized weighting $\mathbf{w}/\vert X\vert$ is the barycentric coordinate of the circumcenter of $S(X)$. The circumcenter lies in the simplex (resp. in its interior) if and only if these coordinates are nonnegative, $\mathbf{w}\geq 0$ (resp. positive, $\mathbf{w}>0$).
\end{proof}

In the following, the terms ``$M$-matrix, Laplacian and nonobtuse" are defined in the proof. 
\begin{proposition}
The following are equivalent for a positive definite metric space $X$.
\begin{enumerate}
\item $K^\dagger$ has nonpositive (resp. negative) off-diagonal entries, equivalently, $c\geq 0$ (resp. $c>0$).
\item $K^\dagger$ is an $M$-matrix (resp. strict $M$-matrix);
\item $K^\dagger$ is the Laplacian matrix of a connected, positively weighted graph (resp. complete graph).
\item $S(X)$ are the vertices of a nonobtuse (resp. acute) simplex.
\end{enumerate}
\end{proposition}
\begin{proof}
(1.$\Leftrightarrow$ 2.) An $M$-matrix is a matrix whose eigenvalues have nonnegative real part and whose off-diagonal entries are non-positive. Since $K^\dagger$ is positive semidefinite when $X$ is positive definite, $K^\dagger$ is indeed an $M$-matrix if and only if its off-diagonal entries are non-positive.

(1.$\Leftrightarrow$ 3.) The Laplacian matrix of a graph $G=(V,E)$ with edge weights $\omega:E\rightarrow \R_{>0}$ is the matrix $L=\sum_{ij\in E}\omega(ij)(\mathbf{e}_i-\mathbf{e}_j)(\mathbf{e}_i-\mathbf{e}_j)^T$. We have established before that $K^\dagger$ can be written in this form, with $\omega(ij)=c_{ij}$, as a consequence of $\spn(\mathbf{1})\subseteq\ker(K^\dagger)$. Thus, if $c>0$ then the matrix $K^\dagger$ is a Laplacian matrix. A graph is connected if and only if $\ker(L)=\spn(\mathbf{1})$ (see, e.g., \cite[\S6]{fiedler_matrices_2011}), which completes the equivalence.

(3.$\Leftrightarrow$ 4.) A simplex is nonobtuse if every angle between facets is at most $\pi/2$ (i.e., nonobtuse). The equivalence between Laplacian matrices of connected graphs $L$ and nonobtuse simplices is due to Fiedler \cite[\S6.2]{fiedler_matrices_2011}; the equivalence is established via (Moore--Penrose pseudoinverse of Laplacian matrix)$\longleftrightarrow$(centered Gram matrix of simplex).

If $c>0$ in (1.) this corresponds to $G$ being the complete graph in (3.) and the simplex having angles smaller than $\pi/2$ (i.e., acute) in (4.). The terminology ``strict $M$-matrix" in (2.) is not standard, but we use it here for an $M$-matrix with negative off-diagonal entries.
\end{proof}

\begin{remark}[Graphs \& discrete curvature]
We highlight in particular the connection with graph Laplacian matrices. Many of the techniques in this article are inspired by the author's previous work in that context; see for instance \cite[\S2--4,\S6]{devriendt_thesis_2022}. Relatedly, we point out that similar objects to $\mathbf{w}$ have recently been defined on graphs in the context of discrete curvature \cite{devriendt_curvature_2022,devriendt_curvature_2024,steinerberger_curvature_2022}; there, $\mathbf{w}>0$ has the natural interpretation of a positive curvature condition. 
\end{remark}

The following result shows that strongly positive definite metric spaces can be parametrized (via $Z^{-1}$) by an open semi-algebraic subset of $\Gl_n(\R)$, the space of real, invertible $n\times n$ matrices.  
\begin{proposition}
A metric space $X$ is strongly positive definite if and only if its similarity matrix $Z$ is invertible with positive magnitude, and satisfies
$$
(Z^{-1}\mathbf{1})_i(Z^{-1}\mathbf{1})_j >(\mathbf{1}^T Z^{-1}\mathbf{1})\cdot (Z^{-1})_{ij} \,\,\,\textup{~and~}\,\,\, (Z^{-1}\mathbf{1})_i(Z^{-1}\mathbf{1})_j >0 \textup{,~for all distinct $i,j\in X$}.
$$
\end{proposition}
\begin{proof}
This follows from the identities $K^\dagger = 2Z^{-1}-2\mathbf{ww}^T/\vert X\vert$ and $\mathbf{w}=Z^{-1}\mathbf{1}$. Introducing $K^\dagger$ and $\mathbf{w}$ in the inequalities above implies that $c>0$ is equivalent to the first inequalities, whereas the second inequalities are equivalent to $\mathbf{w}>0$. For $w_i\cdot w_j>0$ to imply that $\mathbf{w}>0$, we use that $\mathbf{w}$ cannot be all negative because $\vert X\vert=\mathbf{1}^T\mathbf{w}$ is positive.
\end{proof}

As a setup for the next subsection, we establish two further results:

\begin{proposition}\label{prop: strong positive definite asymptotic}
For any metric space $X$ and $t\gg 0$, $tX$ is strongly positive definite.
\end{proposition}
\begin{proof}
The positivity of the weighting $\mathbf{w}$ of $tX$ as $t\gg 0$ was established by Leinster in \cite[Prop. 2.2.6]{leinster_magnitude_2013}. We use the same proof strategy to show $c>0$.  First, let $t$ be sufficiently large such that $tX$ is positive definite, and represent a metric space by its invertible similarity matrix $Z\in\Gl_n(\R)\subset \R^{n\times n}$; here $n=\# X$. On the subset of invertible matrices $\Gl_n(\R)$, the following function is continuous 
$$
\Gl_n(\R)\ni Z\mapsto c_{ij}(Z) = -2\cdot \rm{adj_{ij}}\begin{pmatrix}0&\mathbf{1}^T\\\mathbf{1}&Z\end{pmatrix}/\det\begin{pmatrix}0&\mathbf{1}^T\\\mathbf{1}&Z\end{pmatrix}
$$ 
for any $i\neq j\in X$, where $\rm{adj_{ij}}$ denotes the $(i,j)$th entry of the adjugate matrix. By continuity and $c_{ij}(I)>0$ for the identity matrix $I$, there is a neighbourhood $U\subset \Gl_n(\R)$ of the identity matrix such that $c_{ij}(Z)>0$ for each $Z\in U$. Since $Z(tX)\rightarrow I$ as $t\rightarrow \infty$, we know that $Z(tX)\in U$ and thus that $c_{ij}(Z(tX))>0$ for $t\gg 0$. This completes the proof.
\end{proof}

\begin{proposition}\label{prop: strong positive definite closure}
Strong positive definiteness and $c>0$ are preserved under taking subspaces.
\end{proposition}
\begin{proof}
This follows from Corollaries \ref{cor: magnitude Y} and \ref{cor: c Y}, which express $c'$ and $w'$ for a subspace $Y=X\setminus\{x\}$ in terms of $c$ and $w$ of the metric space $X$, or by repeated application of this argument following the composition property of Schur complements (Proposition \ref{prop: schur complement}). 
\end{proof}

\subsection{Magnitude submodularity}
This section deals with valuative properties of magnitude. We show that two magnitude-based set functions on the subspaces of a strongly positive definite metric space $X$ are submodular; recall that a function $f$ on the subsets of $X$ is called \emph{strictly submodular} if it satisfies\footnote{As noted, this is a different but equivalent definition from the one given in the introduction.}
\begin{equation}\label{def: submodularity}
f(Y)-f(Y\backslash\{y\}) < f(Y\backslash\{x\})-f(Y\backslash\{x,y\}), \text{~for all $Y\subseteq X$ and distinct $x,y\in Y$},
\end{equation}
and $f$ is called \emph{increasing} (resp. \emph{decreasing}) if it is increasing (resp. decreasing) with respect to the subset partial order on $2^X$. We start with the inverse of magnitude as a set function:
\begin{theorem}\label{thm: inverse submodularity}
Let $X$ be strongly positive definite. Then the function
$$
f: Y\longmapsto 
\begin{cases}
-\vert Y\vert^{-1} \textup{,~if $Y\neq \emptyset$}\\ 
\,\,\,\,\,\,\alpha \,\,\,\quad\textup{,~if $Y=\emptyset$}
\end{cases}
$$
is increasing if $\alpha<-1$ and strictly submodular if $\alpha< -\tfrac{3}{2}$. In particular, $f$ is strictly submodular on the subspaces of $tX'$, for any $X'$ and $t\gg 0$ and $\alpha<-\tfrac{3}{2}$.
\end{theorem}
\begin{proof}
Because strong positive semidefiniteness is conserved for subspaces (Proposition \ref{prop: strong positive definite closure}), we may assume without loss of generality that $Y=X$ when checking the submodularity inequalities \eqref{def: submodularity}. For $\#X = 1$ the statement is true since $\vert X\vert=1$ and \eqref{def: submodularity} holds vacuously; we distinguish the two remaining cases, $\#X=2$ and $\# X>2$, and start with the latter.

($\#X>2$): We can rewrite the submodularity inequalities for $f$ as
\begin{align*}
-\vert X\vert^{-1} + \vert X\backslash\{x\}\vert^{-1} \,&<\, -\vert X\backslash\{y\}\vert^{-1} + \vert X\backslash\{x,y\}\vert^{-1}
\\
\iff \frac{2w_x^2}{\bar{c}_x\cdot \vert X\vert^2}\,\, &<\,\,\frac{2{w'}_x^2}{\bar{c}'_x\cdot \vert X\backslash\{y\}\vert^2} \quad\quad\quad\quad\quad\text{~(Corollary \ref{cor: magnitude Y})},
\end{align*}
where $c,w$ are defined on $X$ and $c',w'$ on $X\backslash\{y\}$. For strongly positive definite $X$, Corollaries \ref{cor: magnitude Y} and \ref{cor: c Y} imply $\bar{c}'_x<\bar{c}_x$ and $w'_x/\vert X\backslash\{y\}\vert>w_x/\vert X\vert$, which confirms the inequality above.

($\#X=2$): We can rewrite the submodularity inequality for $f$ as
\begin{align*}
-\vert X\vert^{-1} + \vert X\backslash\{x\}\vert^{-1}\, &< -\,\vert X\backslash\{y\}\vert^{-1}-\alpha
\\
\iff 1-\frac{1+\exp(-d(x,y))}{2} \,&<\, -1-\alpha \quad\quad\quad\quad\text{~(Example \ref{ex: 2-point})}
\\
\Longleftarrow \alpha \,&<\,\tfrac{-3}{2} 
\end{align*}
This confirms strict submodularity for $\#X=2$ if $\alpha<-2/3$, and thus completes the proof for strongly positive definite $X$. By Corollary \ref{cor: magnitude Y}, we know that $f$ is increasing for nonempty subsets and that $f(Y)\geq -1$ with equality for $\# Y =1$; it follows that $f$ is increasing on $2^X$ if $\alpha<-1$. The asymptotic result follows from Proposition \ref{prop: strong positive definite asymptotic}.
\end{proof}

Equivalently, one can say that $-f$ in Theorem \ref{thm: inverse submodularity} is strictly supermodular and decreasing, for appropriate $\alpha$. To end this section, we prove a second submodularity result. The function considered in the theorem below is inspired by the asymptotic expression in Theorem \ref{thm: asymptotics of q} for $q(tX)=\#X-\vert X\vert$ in terms of $R(S_t)^2$.
\begin{theorem}\label{thm: shifted submodularity}
Let $X$ be any metric space and $t\gg 0$. Then the function
\begin{equation*}
f:Y \longmapsto \begin{cases}
\frac{m-\vert tY\vert}{m^2} + \frac{m-1}{m} \textup{~~~if $m:=\# Y \neq 0$}\\
\quad\quad\quad \alpha \quad\quad\quad\textup{~if $Y=\emptyset$}
\end{cases}
\end{equation*}
is increasing if $\alpha<\tfrac{1}{2}$ and strictly submodular if $\alpha<-\tfrac{1}{2}$.
\end{theorem}
\begin{proof}
We follow the strategy from the proof of Proposition \ref{prop: strong positive definite asymptotic} and represent a metric space $X$ by its invertible similarity matrix $Z\in \Gl_n(\R)$. We consider three families of functions: the family $F_{Y,x,y}$, labeled by $x\neq y\in Y\subseteq X$, the family $G_{x,y}$ labeled by $x\neq y\in X$, and the family $H_{Y',Y}$ labeled by $Y'\subset Y\subseteq X$; these are defined as
\begin{align*}
F_{Y,x,y} &= f(Y)-f(Y\backslash\{x\})-f(Y\backslash\{y\})+f(Y\backslash\{x,y\})
\\
G_{x,y} &= f(\{x,y\})-f(\{x\})-f(\{y\})+f(\emptyset)
\\
H_{Y',Y} &= f(Y')-f(Y)
\end{align*}
Since $\vert Y\vert = -\det\left(\begin{smallmatrix}0&\mathbf{1}^T\\\mathbf{1}&Z_{YY}\end{smallmatrix}\right)/\det(Z)$ by Proposition \ref{prop: magnitude determinant}, these are continuous functions on $\Gl_n(\R)$. Next, we confirm that each of these functions evaluates negatively for the discrete metric space; in that case, $Z=I$ and thus $\vert Y\vert=m$ and $f(Y)=(m-1)/m$. We find
\begin{align*}
F_{Y,x,y}(I) &= \tfrac{m-1}{m} -2\tfrac{m-2}{m-1} + \tfrac{m-3}{m-2} = \tfrac{-2}{m(m-1)(m-2)}<0
\\
G_{x,y}(I) &= \tfrac{1}{2} +\alpha
\\
H_{Y',Y}(I) &= \begin{cases}
\tfrac{m'-1}{m'}-\tfrac{m-1}{m}=\tfrac{m'-m}{m\cdot m'}<0 &\text{~if $Y'\neq\emptyset$}\\
\alpha-\tfrac{m-1}{m} \leq \alpha - \tfrac{1}{2} &\text{~if $Y'=\emptyset$}
\end{cases}.
\end{align*}
Thus $F(I)<0$, $G(I)<0$ if $\alpha<-\tfrac{1}{2}$ and $H(I)<0$ if $\alpha<\tfrac{1}{2}$; fix a choice of $\alpha$. By continuity of $F,G,H$ and negativity at the identity matrix $I$, there is a neighbourhood $U\subset \Gl_n(\R)$ of $I$, such that $F,G,H$ are negative on each $Z\in U$. Since $Z(tX)\rightarrow I$ as $t\rightarrow\infty$, we thus know that $Z(tX)\in U$ for $t\gg 0$ and thus that $F,G,H$ are negative. These are the strict submodularity inequalities ($F,G$) and the increasing inequalities ($H$), which completes the proof.
\end{proof}

\section*{Acknowledgments} I thank Mark Meckes for helpful discussions and encouragements. Part of this article was written during a research stay at the Max Planck Institute of Molecular Cell Biology and Genetics in Dresden. For the purpose of open access, the author has applied a CC BY public copyright licence to any author accepted manuscript arising from this submission.
\bibliographystyle{abbrv}
\bibliography{bibliography.bib}

\bigskip
\bigskip
\noindent {\bf Authors' address:}	

\smallskip		
\noindent Karel Devriendt, University of Oxford
\hfill \url{devriendt@maths.ox.ac.uk}
\end{document}